\documentclass{article}
\pdfoutput=1
\title{A novel discrete variational derivative method\\ using ``average-difference methods''}
\author{Daisuke Furihata$^{1)}$, Shun Sato$^{2),}\footnote{Corresponding author, E-mail: shun\_sato@mist.i.u-tokyo.ac.jp}$ , Takayasu Matsuo$^{2)}$\\
$1)$ Osaka University \hspace{5pt} $2)$ The University of Tokyo}
\date{May, 2016}

\usepackage[dvipdfmx]{graphicx}
\usepackage{amsmath}
\usepackage{amsfonts}
\usepackage{color}
\usepackage{mathtools}
\mathtoolsset{showonlyrefs}

\newcommand{\rd}{\mathrm{d}}

\newcommand{\RR}{\mathbb{R}}

\newcommand{\ud}[2]{u^{\left( #1 \right)}_{ #2 }}

\newcommand{\fd}{\delta^+}

\newcommand{\cd}{\delta^{\langle 1 \rangle}}
\newcommand{\dps}{\delta_{\mathrm{PS}}}

\newcommand{\fa}{\mu^+}

\newcommand{\im}{\mathrm{i}}

\newtheorem{theorem}{Theorem}
\newtheorem{proposition}{Proposition}

\newtheorem{proof}{Proof}

\usepackage{latexsym}
\def\qed{\hfill $\Box$}
\begin{document}

\maketitle

\begin{abstract}
We consider structure-preserving methods for conservative systems, 
which rigorously replicate the conservation property yielding better numerical solutions. 
There, 
corresponding to the skew-symmetry of the differential operator, 
that of difference operators is essential to the discrete conservation law. 
Unfortunately, however, when we employ the standard central difference operator, 
the simplest one, 
the numerical solutions often suffer from undesirable spatial oscillations. 
In this letter, we propose a novel ``average-difference method,'' 
which is tougher against such oscillations, 
and combine it with an existing conservative method. 
Theoretical and numerical analysis in the linear case show the superiority of the proposed method. 
\end{abstract}

\section{Introduction}

In this letter, we consider the numerical integration of the partial differential equation (PDE) in the form
\begin{equation}
u_{tx} = \frac{\delta G}{\delta u}, \qquad \mathcal{H} (u) := \int^L_0 G(u,u_x,\dots) \rd x,
\label{eq_nlKG}
\end{equation}
where subscripts $ t $ or $x$ denote the partial differentiation with respect to $ t $ or $x$, 
and $ \delta G / \delta u $ is 
the variational derivative of $G$. 
We assume the periodic boundary condition $ u(t,x+L ) = u(t,x) \ ( \forall t \in \RR_+ :=[0,+\infty), \  \forall x \in \RR ) $, where $ L \in \RR_+ $ is a constant.  
When the derivatives of $u$ do not appear in $ \mathcal{H} $, 
the equation~\eqref{eq_nlKG} is called the (nonlinear) Klein--Gordon equation in light-cone coordinates. 
Moreover, the class of PDEs in the form~\eqref{eq_nlKG} is closely related to 
the Ostrovsky equation~\cite{O1978}, the short pulse equation~\cite{SW2004}, etc. 
For their numerical treatments, due to the possible indefiniteness caused by the spatial derivative in the left-hand side, 
it seems a systematic numerical framework for~\eqref{eq_nlKG} is yet to be investigated, 
though a few exceptions for specific cases can be found (see, e.g., \cite{YMS2010,MYM2012}).

In this letter, we focus on a certain class of conservative methods. 
Under the periodic boundary condition, the target equation~\eqref{eq_nlKG} 
has the conserved quantity~$ \mathcal{H} $: 
\begin{align}
\frac{\rd \mathcal{H}}{\rd t}
&= \int^L_0 \frac{\delta G}{\delta u} u_t \, \rd x = \int^L_0 \hspace{-5pt} u_{tx} u_t \, \rd x \notag \\
&= \left[ \frac{1}{2}  u_t^2 \right]^L_0 - \int^L_0 u_t u_{tx} \rd x = 
- \int^L_0 u_t u_{tx} \rd x = 0. \hspace{6pt}
\label{eq_conservation}
\end{align}
Note that the skew-symmetry of the differential operator $ \partial_x := \partial / \partial x $ 
is crucial here. 
A numerical scheme is called conservative when it replicates such a conservation property 
(see, e.g., \cite{FM2011,CGMMOOQ2012}). 
The numerical solutions obtained by such schemes are often more stable than those of general-purpose methods. 
There, the crucial point for the discrete conservation law is 
the skew-symmetry of difference operator, which corresponds to that of the differential operator; 
when one tries to construct a conservative finite-difference scheme for the equation~\eqref{eq_nlKG}, 
the differential operator $ \partial_x $ in left-hand side must be replaced by one of the skew-symmetric difference operators, for example, the central difference operators, 
the compact finite difference operators~(see, e.g., Kanazawa--Matsuo--Yaghchi~\cite{KMY2012}), 
and the Fourier-spectral difference operator (see, e.g., Fornberg~\cite{Fornberg}). 
This is intrinsically indispensable, at least to the best of the present authors' knowledge. 
This, however, at the same time, leads to an undesirable side effect that 
the numerical solutions tend to suffer from spatial oscillations. 

In this letter, to work around this technical difficulty, 
we propose a novel ``average-difference method,'' 
which is tough against such undesirable spatial oscillations. 
A similar method has been, in fact, 
already investigated by Nagisa~\cite{NagisaMT}. 
However, he used this method for advection-type equations, 
and concluded the method was unfortunately not more advantageous than existing methods. 
In this letter, we instead construct an average-difference method for the PDE~\eqref{eq_nlKG}, 
and combine it with the idea of conservation mentioned above. 
Then we compare the proposed and existing methods in the case of the linear Klein--Gordon equation, 
which is the simplest case with $ G (u) = u^2/2 $. 
As a result, 
the average-difference type method is successfully superior to the existing methods 
in view of the phase speed of each frequency component.


\section{The standard conservative method}
\label{sec_standard}

The conservative scheme for the PDE~\eqref{eq_nlKG} 
can be constructed in the spirit of discrete variational derivative method (DVDM) (see, the monograph~\cite{FM2011} for details).  
There, one utilizes the concept of the ``discrete variational derivative'' and skew-symmetric difference operators. 
The symbol $ \ud{m}{k} $ denotes the approximation $ \ud{m}{k} \approx u ( m \Delta t , k \Delta x )  \ ( m =  0,\dots, M; k \in \mathbb{Z} ) $, 
where $ \Delta t $ and $ \Delta x \, (:= L/K)$ are the temporal and spatial mesh sizes, respectively. 
Here, we assume the discrete periodic boundary condition $ \ud{m}{k+K} = \ud{m}{k} \ ( k \in \mathbb{Z} ) $, 
and thus, we use the notation $ \ud{m}{} := ( \ud{m}{1} , \dots ,\ud{m}{K} )^{\top} $. 
Let us introduce the spatial central difference operator $ \cd_x $ and the temporal forward difference operator $ \fd_t $: 
\begin{align}
\cd_x \ud{m}{k} &= \frac{ \ud{m}{k+1} - \ud{m}{k-1} }{2\Delta x}, &
\fd_t \ud{m}{k} &= \frac{\ud{m+1}{k} - \ud{m}{k}}{\Delta t}.
\end{align}

The discrete counterpart $ \mathcal{H}_{\rd} $ of the functional $\mathcal{H} $ can be defined as 
\begin{equation}
\mathcal{H}_{\rd} \left( \ud{m}{} \right) := \sum_{k=1}^K G_{\rd} \left( \ud{m}{k} \right) \Delta x,
\end{equation}
where $ G_{\rd} ( \ud{m}{k} ) $ is an appropriate approximation of $ G (u,u_x,\dots ) $. 
Then, the discrete variational derivative $ \delta G_{\rd} / \delta ( \ud{m+1}{} ,\ud{m}{} )_k $ 
is defined as a function satisfying 
\begin{equation}
\fd_t \mathcal{H}_{\rd} \left( \ud{m}{k} \right) = \sum_{k=1}^K \frac{\delta G_{\rd}}{ \delta \left( \ud{m+1}{}, \ud{m}{} \right)_k } \fd_t \ud{m}{k} \Delta x. 
\label{eq_dvd_prop}
\end{equation}
For the construction of such one, see~\cite{FM2011}. 
By using the discrete variational derivative, 
we can construct a conservative scheme
\begin{equation}
\cd_x \fd_t \ud{m}{k} = \frac{\delta G_{\rd}}{ \delta \left( \ud{m+1}{}, \ud{m}{} \right)_k }. 
\label{eq_dvdm_cd}
\end{equation}
As stated in the introduction, 
the key ingredient here is the skew-symmetry of the central difference operator. 

\begin{proposition} 
Suppose the numerical scheme~\eqref{eq_dvdm_cd} 
has a solution $ \ud{m}{k} $ under the periodic boundary condition. 
Then, it satisfies $ \mathcal{H}_{\rd} ( \ud{m+1}{} ) = \mathcal{H}_{\rd} ( \ud{m}{} ) $. 
\end{proposition}

\begin{proof}
{\em 
Thanks to the definition~\eqref{eq_dvd_prop} of the discrete variational derivative,  
we can follow the line of the discussion~\eqref{eq_conservation} as follows: 
\begin{align}
\fd_t \mathcal{H}_{\rd} \left( \ud{m}{} \right)
&= \sum_{k=1}^K \frac{\delta G_{\rd}}{ \delta \left( \ud{m+1}{}, \ud{m}{} \right)_k } \fd_t \ud{m}{k} \Delta x \\
&= \sum_{k=1}^K \left( \cd_x \fd_t \ud{m}{k} \right) \fd_t \ud{m}{k} \Delta x, 
\end{align}
whose right-hand side vanishes due to the skew-symmetry of the central difference operator~$ \cd_x $: 
\begin{equation}
\sum_{k=1}^K u_k \cd_x v_k \Delta x = - \sum_{k=1}^K \left( \cd_x u_k \right) v_k \Delta x. 
\end{equation}
holds for any $ u, v \in \RR^K $. }
\qed
\end{proof}

The discrete conservation law can also be proved similarly for the other skew-symmetric difference operators.

\section{``Average-difference method''}
\label{sec_ad}

In this section, we propose the novel method. 
There, instead of the single skew-symmetric difference operator, 
we employ the pair of the forward difference and average operators: 
\begin{align}
\fd_x \ud{m}{k} &= \frac{\ud{m}{k+1} - \ud{m}{k}}{\Delta x}, &
\fa_x \ud{m}{k} &= \frac{\ud{m}{k+1}+\ud{m}{k}}{2}.
\end{align}
The average-difference method for the equation~\eqref{eq_nlKG} can be written in the form 
\begin{equation}
\fd_x \fd_t \ud{m}{k} = \fa_x \frac{ \delta G_{\rd} }{ \delta \left( \ud{m+1}{}, \ud{m}{} \right)_k }. 
\label{eq_dvdm_ad}
\end{equation}
The name ``average-difference'' comes from the idea of approximating $ \partial_x $ with the pair of $ ( \fd_x , \fa_x ) $; 
this makes sense for more general PDEs, 
and thus is independent of any conservation properties. 
Still, in this letter we focus on~\eqref{eq_nlKG} and~\eqref{eq_dvdm_ad}.  
 
Although it is constructed in the spirit of DVDM, 
now the forward difference operator $ \fd_x $ loses the apparent skew-symmetry, 
and accordingly, 
the proof of the discrete conservation law becomes unobvious. 
A similar proof can be found in Nagisa~\cite{NagisaMT}. 

\begin{theorem}
Suppose the average-difference method~\eqref{eq_dvdm_ad} 
has a solution $ \ud{m}{k} $ under the periodic boundary condition. 
Then, it satisfies $ \mathcal{H}_{\rd} ( \ud{m+1}{} ) = \mathcal{H}_{\rd} ( \ud{m}{} ) $. 
\label{thm_dvdm_ad}
\end{theorem}

\begin{proof}
{\em 
By using the definition~\eqref{eq_dvd_prop} of the discrete variational derivative, we see that
\begin{equation}
\fd_t \mathcal{H}_{\rd} \left( \ud{m}{} \right)
= \sum_{k=1}^{K} \frac{ \delta G_{\rd} }{\delta ( \ud{m+1}{} , \ud{m}{} )_k } \fd_t \ud{m}{k} \Delta x.  
\end{equation}
Here, for brevity, we introduce the notation
\[ a_k = \frac{ \delta G_{\rd} }{\delta ( \ud{m+1}{} , \ud{m}{} )_k }, \qquad b_k = \fd_t \ud{m}{k}. \]
Note that the equation~\eqref{eq_dvdm_ad} implies the relation $ \fd_x b_k = \fa_x a_k $. 
By using the identity
\begin{align}
 \frac{\alpha^+ \beta^+ + \alpha \beta}{2} 
= \left( \frac{\alpha^+ + \alpha}{2} \right) \left( \frac{\beta^+ + \beta}{2} \right) + \frac{1}{4} \left( \alpha^+ - \alpha \right) \left( \beta^+ - \beta \right), 
\end{align}
which holds for any $ \alpha,  \alpha^+ , \beta, \beta^+ \in \RR $, 
we see 
\begin{align*}
\sum_{k=1}^K a_k b_k 
&= \sum_{k=1}^K \frac{ a_{k+1} b_{k+1} + a_k b_k }{2} \\
&= \sum_{k=1}^K \left( \left( \fa_x a_k \right) \left( \fa_x b_k \right) + \frac{\Delta x^2}{4}  \left( \fd_x a_k \right) \left( \fd_x b_k \right) \right) \\
&= \sum_{k=1}^K \left( \left( \fd_x b_k \right) \left( \fa_x b_k \right) + \frac{\Delta x^2}{4} \left( \fd_x a_k \right) \left( \fa_x a_k \right) \right) \\
&= \sum_{k=1}^K \left( \frac{1}{2} \left( \fd_x b_k^2  + \frac{\Delta x^2}{4} \fd_x a_k^2 \right) \right) =0,
\end{align*}
which proves the theorem. }
\qed
\end{proof}

\section{Analysis in the linear Klein--Gordon equation}
\label{sec_anal_lkg}


In order to conduct a detailed analysis, 
we consider the simplest case, the linear Klein--Gordon equation 
\begin{align}
u_{tx} &= \frac{ \delta G }{ \delta u } = u, &
\mathcal{H} (u) &:= \frac{1}{2} \int^{2 \pi}_0 u^2 \rd x
\label{eq_lKG_vf}
\end{align}
under the periodic domain with the period $ L = 2 \pi $. 
The exact solution of the linear Klein--Gordon equation~\eqref{eq_lKG_vf} can be formally written in the form
\begin{equation}
u ( t,x) = \sum_{n \in \mathbb{Z} \setminus \{ 0 \}} a_n \exp \left( - \im \frac{t}{n} \right) \exp \left( \im n x \right), 
\end{equation}
where $ \im $ is the imaginary unit, and $ a_n \in \mathbb{C} $ is determined by the initial condition $ u (0,x) = u_0 (x) $: 
\begin{equation}
a_n = \frac{1}{ 2 \pi} \int^{2 \pi}_0 u_0 (x) \exp \left( - \im n x \right) \rd x. 
\end{equation}
In view of the superposition principle, we consider the single component \linebreak $ \exp ( - \im t / n ) \exp ( \im n x ) $ ($ n \in \mathbb{Z} \setminus \{ 0 \}$).

\subsection{Comparison of phase speeds}

In order to clarify the difference between the standard conservative method and proposed method, 
we consider the following three semi-discretizations
\begin{align}
\cd_x \dot{u}_k &= u_k, \label{eq_lkg_cd}\\
\dps \dot{u}_k &= u_k, \label{eq_lkg_ps}\\
\fd_x \dot{u}_k &= \fa_x u_k, \label{eq_lkg_ad}
\end{align}
where $ u_k (t) \approx u(t, k\Delta x) $ for $ k \in \mathbb{Z} $ ($ u_{k+K} = u_k $). 
Here, $ \dps $ denotes the Fourier-spectral difference operator, i.e., 
\begin{equation}
\dps u_k := \begin{cases} {\displaystyle \frac{1}{\sqrt{K}} \sum_{j=- \frac{K-1}{2}}^{\frac{K-1}{2}}  \im j \exp \left( \frac{2 \pi \im jk}{K} \right) \tilde{u}_j } & (K:\mbox{odd}), \\ {\displaystyle \frac{1}{\sqrt{K}} \sum_{j=- \frac{K-2}{2}}^{\frac{K-2}{2}}  \im j \exp \left( \frac{2 \pi \im jk}{K} \right) \tilde{u}_j } & (K:\mbox{even}), \end{cases}
\label{eq_dps}
\end{equation}
where $ \tilde{u}_k $ is obtained by the discrete Fourier transform: 
\begin{equation}
\tilde{u}_k := \frac{1}{\sqrt{K}} \sum_{j=1}^K \exp \left( - \frac{2 \pi \im k j }{K} \right) \ud{m}{j}. 
\label{eq_dft}
\end{equation}
Note that, the implicit midpoint method for the semi-discretizations above 
coincide with the numerical schemes constructed in the previous sections. 

We consider the solution of the semi-discretizations above in the form $ u_k = \exp ( \im c_n t ) \exp ( \im n k \Delta x) $ ($ c_n \in \RR$)
for each $ n \in \mathbb{Z} \setminus \{ m \in \mathbb{Z} \mid 2 m / K \notin \mathbb{Z} \} $,
which gives an exact solutions of~\eqref{eq_lkg_cd}, \eqref{eq_lkg_ps}, and \eqref{eq_lkg_ad} with appropriate choices of $ c_n$. 
For the central difference scheme~\eqref{eq_lkg_cd}, we see
\begin{equation}
c^{\mathrm{CD}}_n = - \frac{ \Delta x }{ \sin n \Delta x }.
\end{equation}
If we employ the Fourier-spectral difference operator instead of the central difference, 
we see 
\begin{equation}
c^{\mathrm{PS}}_n = - \frac{1}{n} \qquad ( | n | < K/2 ),
\end{equation}
and $ c_{n+K} = c_{n} $ holds for any $ n \in \mathbb{Z} \setminus \{ m \in \mathbb{Z} \mid 2 m / K \notin \mathbb{Z} \} $. 
For the average-difference scheme~\eqref{eq_lkg_ad}, we obtain 
\begin{equation}
c^{\mathrm{AD}}_n = - \frac{\Delta x}{2  \tan ( n \Delta x/2 ) }.
\end{equation}

The phase speeds $ c_n $ corresponding to each numerical scheme are summarized in Fig.~\ref{fig_speed} ($ K = 65 $). 
As shown in Fig.~\ref{fig_speed}, the phase speed of the central difference scheme~\eqref{eq_lkg_cd}
are falsely too fast for high frequency components ($ n \approx K/2 $). 
On the other hand, the error of the phase speeds of the average-difference method are 
much smaller. 

\begin{figure}
\centering
\includegraphics{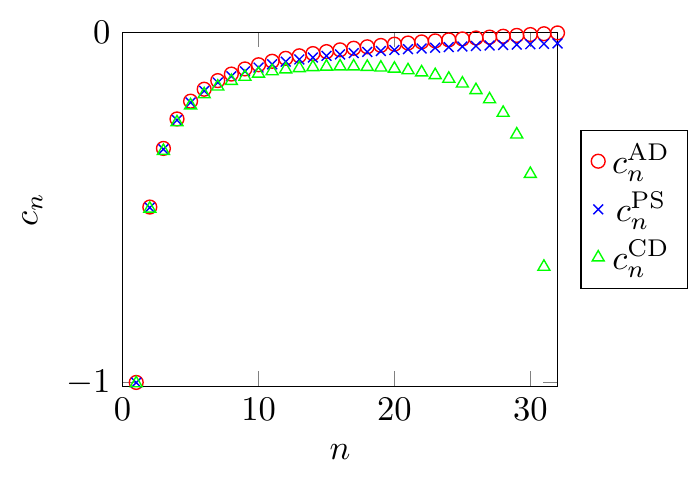}
\caption{The phase speeds for each $n =1,2,\dots,32 $ ($K:=65$). The red circles, blue crosses, and green triangles correspond to 
the average-difference, Fourier-spectral difference, and central difference schemes, respectively. }
\label{fig_speed}
\end{figure}

\subsection{Numerical experiment}

In this section, we conduct a numerical experiment under the periodic boundary condition $ u(t,x+2 \pi) = u (t,x) $ with the initial condition 
\begin{equation}
u_0(x) = \begin{cases}
1 \quad & ( \pi / 2 < x < 3 \pi / 2 ), \\
-1 \quad & (0\le x \le \pi/2 \text{ or } 3 \pi/2 \le x < 2 \pi).
\end{cases}
\end{equation}
The corresponding solution can be formally written in 
\begin{equation}
u ( t,x) = \sum_{n=1}^{\infty} \left( - \frac{4}{ n \pi} \sin \frac{ n\pi }{2} \right) \cos \left( n x - \frac{t}{n} \right).
\end{equation}

Figures~\ref{fig_step_CD}, \ref{fig_step_PS}, and \ref{fig_step_AD} show the 
numerical solutions of the central difference scheme~\eqref{eq_lkg_cd}, 
the Fourier-spectral difference scheme~\eqref{eq_lkg_ps}, 
and the average-difference method~\eqref{eq_lkg_ad}, respectively (the temporal discretization: implicit midpoint method). 
As shown in Fig.~\ref{fig_step_CD}, the central difference scheme suffers from the spatial oscillation, 
whereas the other schemes reproduce the smooth profiles until $ t = 1$. 
The cause of this difference lies on the discrepancy in phase speeds of high frequency components (Fig.~\ref{fig_speed}). 

\begin{figure}
\centering
\includegraphics{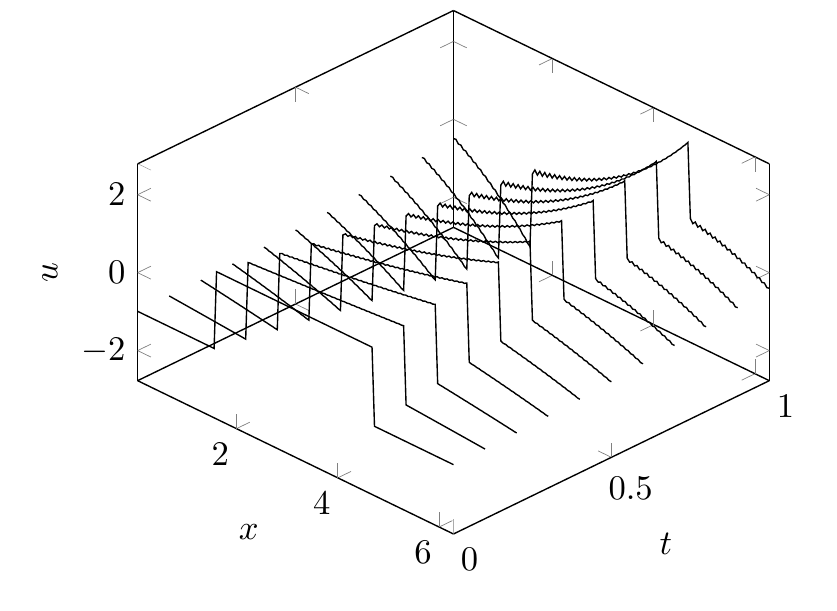}
\caption{The numerical solution of the central difference scheme~(7) ($K=129$, $\Delta t = 0.01$). }
\label{fig_step_CD}
\end{figure}
\begin{figure}
\centering
\includegraphics{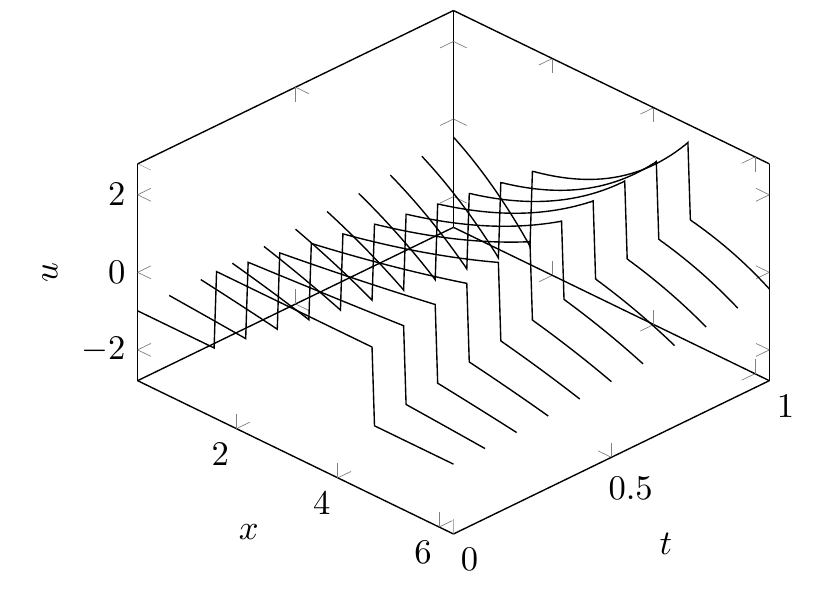}
\caption{The numerical solution of the Fourier-spectral difference scheme~(8) ($K=129$, $\Delta t = 0.01$).}
\label{fig_step_PS}
\end{figure}
\begin{figure}
\centering
\includegraphics{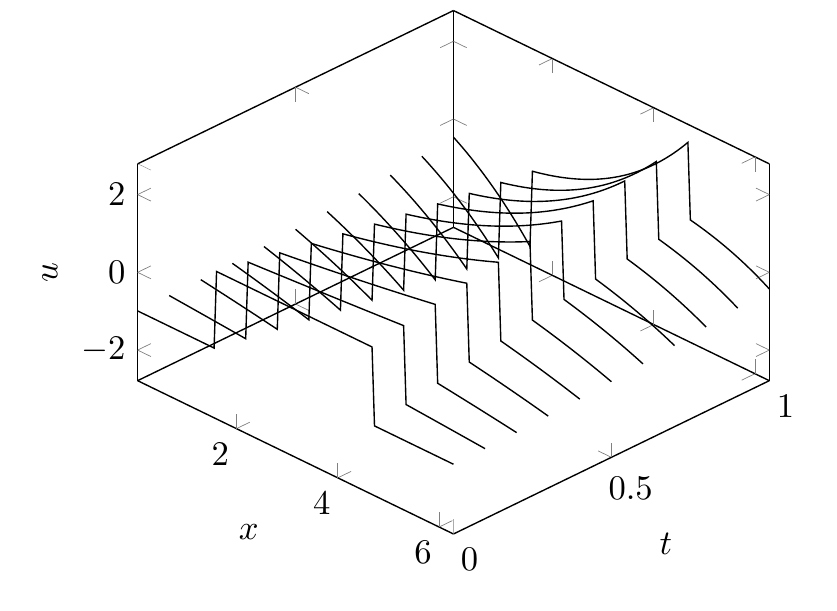}
\caption{The numerical solution of the average-difference scheme~(9) ($K=129$, $\Delta t = 0.01$).}
\label{fig_step_AD}
\end{figure}

However, as shown in Fig.~\ref{fig_comp}, which shows the numerical solutions of each schemes at $ t = 50 $, 
the Fourier-spectral scheme also suffers from the undesirable spatial oscillation, 
whereas the proposed method, average-difference method reproduces a better profile. 
Moreover, the values of error 
\[ \sum_{k=1}^K \left( \ud{M}{k} - u(M\Delta t, k \Delta x) \right)^2 \Delta x \]
 at $M=5000$ (i.e., $t=50$) for the numerical solutions of 
central difference scheme, Fourier-spectral difference scheme, and the proposed method
are $ 0.1940 $, $ 0.0611$, and $ 0.0575 $, respectively. 
This could be attributed to the fact that the Fourier-spectral difference can be regarded as a higher-order 
central difference, 
and thus should share the same property to a certain extent. 

\begin{figure}
\centering
\includegraphics{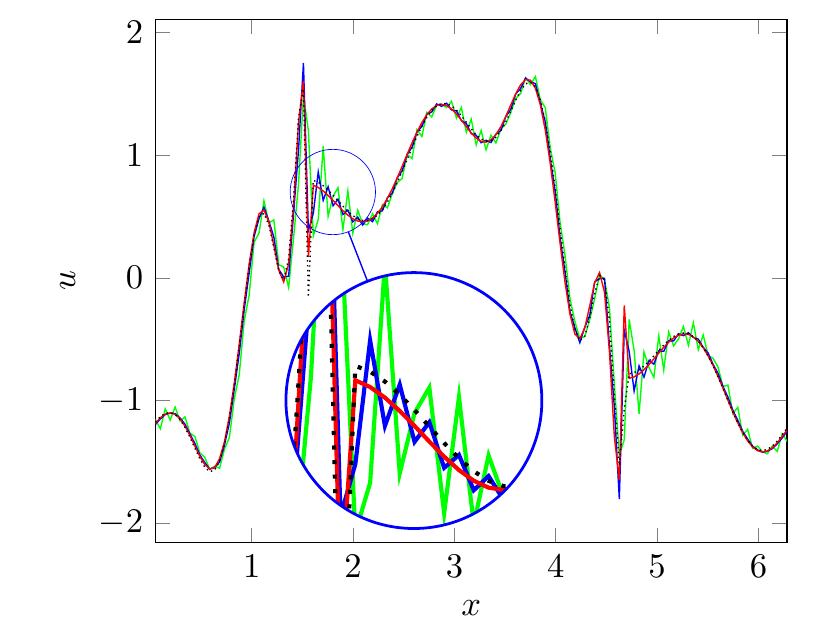}
\caption{The comparison of the numerical solutions at $ t = 50 $ ($K=129$, $\Delta t = 0.01$). 
The black dotted line represents the exact solution. 
The green, blue, and red solid lines denote the numerical solution of the central difference scheme, the Fourier-spectral difference scheme, and the average-difference method, respectively. }
\label{fig_comp}
\end{figure}

\section{Concluding remarks}
\label{sec_conc}


The results above can be extended in several ways. 
First, instead of the cumbersome proof in Theorem~\ref{thm_dvdm_ad}, we can introduce the concept of generalized skew-symmetry, by which a more sophisticated ``average-difference'' version of the DVDM could be given.
Second, we should try more general PDEs to see to which extent the new DVDM is advantageous.
Finally and ultimately, we hope to construct a systematic numerical framework for (1), based on the above observations.
The authors have already got some results on these issues, which will be reported somewhere soon.

\section*{Acknowledgments}
This work was partly supported by JSPS KAKENHI Grant Numbers 25287030, 26390126, and 15H03635, and by CREST, JST. 
The second author is supported by the JSPS Research Fellowship for Young Scientists.

\end{document}